\documentclass[12pt,a4paper]{amsart}

\usepackage{amsmath,amsfonts,amssymb,mathtools,extarrows}

\usepackage[hmargin=2.5cm,vmargin=2.5cm]{geometry}
\usepackage[dvipsnames]{xcolor}	
\usepackage{hyperref}
\usepackage{nameref,zref-xr}     

\hypersetup{
colorlinks=true, linktocpage=true, pdfstartpage=1, pdfstartview=FitV,breaklinks=true, pdfpagemode=UseNone, pageanchor=true, pdfpagemode=UseOutlines,plainpages=false, bookmarksnumbered, bookmarksopen=true, bookmarksopenlevel=1,hypertexnames=true, pdfhighlight=/O,urlcolor=webbrown, linkcolor=RoyalBlue, citecolor=ForestGreen}

\usepackage{enumerate}      

\usepackage{scalerel}         
\usepackage{stmaryrd}

\usepackage[all]{xy}

\usepackage[textsize=footnotesize,textwidth=0.85in]{todonotes}
\presetkeys%
    {todonotes}%
    {color=Apricot}{}%
\newcommand{\oM}{\overline{\mathcal M}}

\def\oM{{\overline{\mathcal{M}}}}

\newcommand{\DR}{\mathrm{DR}}

\newcommand{\PSSRT}{\mathrm{PSSRT}}

\newcommand{\ZZ}{\mathbb{Z}}
\newcommand{\QQ}{\mathbb{Q}}

\newcommand{\ev}{\mathrm{ev}}

\newtheorem{theorem}{Theorem}[section]

\theoremstyle{remark}

\theoremstyle{definition}

\usepackage{color}

\newcommand{\bD}{{\mathbb{D}}}

\numberwithin{equation}{section}

\begin{document}

\title[]{The master relation for polynomiality and equivalences of integrable systems}

\author[X.~Blot]{Xavier Blot}
\address{X.~B.: Korteweg-de Vriesinstituut voor Wiskunde, Universiteit van Amsterdam, Postbus 94248, 1090GE Amsterdam, Nederland}
\email{x.j.c.v.blot@uva.nl}	

\author[A.~Sauvaget]{Adrien Sauvaget}
\address{A.~S.: CNRS, Université de Cergy-Pontoise, Laboratoire AGM, UMR 8088, 2 av. Adolphe Chauvin 95302,
	Cergy-Pontoise Cedex, France}
\email{adrien.sauvaget@math.cnrs.fr} 

\author[S.~Shadrin]{Sergey Shadrin}
\address{S.~S.: Korteweg-de Vriesinstituut voor Wiskunde, Universiteit van Amsterdam, Postbus 94248, 1090GE Amsterdam, Nederland}
\email{s.shadrin@uva.nl}	

\begin{abstract} We prove the so-called master relation in the tautological ring of the moduli space of curves that implies polynomial properties of the Dubrovin--Zhang hierarchies associated to different versions of cohomological field theories as well as their equivalences to the corresponding double ramification hierarchies.
\end{abstract}

\maketitle

\section{Introduction}

The main goal of the paper is to prove the so-called {\em master relation} in $R^*(\oM_{g,n+m})$, $g\geq 0$, $n,m\geq 1$, $2g-2+n+m>0$, conjectured in~\cite[Conjecture 3.4]{BLS}. This master relation is proven in \emph{op. cit.} in the Gorenstein quotient and used to imply the strong form of the {\em DR/DZ equivalence} conjecture due to Buryak-Dubrovin-Gu\'er\'e-Rossi~\cite{BDGR1,BDGR20} (see also the foundational work of Buryak~\cite{Bur} for an earlier weaker version of the DR/DZ equivalence conjecture).

The master relation in the Gorenstein quotient is sufficient for the most important applications to integrable systems associated with ordinary Cohomological field theories (CohFTs), partial CohFTs, (ordinary or partial), and F-CohFTs. Indeed, the examples considered in the literature in the context of integrable systems are constructed in terms of the tautological classes. Yet, it was believed for a long time that polynomiality and equivalence properties of integrable systems of topological type are governed by actual relations in the tautological ring of the moduli space of curves. To this end, a number of conjectures were made, see~\cite[Conjecture 2.5]{BGR19} and~\cite[Conjectures 1, 2, 3]{BS22}. All these conjectures were shown through a sequence of works~\cite{BS22,BLRS,BLS,BLS-Omega} to follow from the master relation. Therefore, the present paper settles these conjectures as well. 

Despite the fact that the master relation that we prove in this paper is of crucial importance for integrable systems, we won't survey these applications as it is already done in a very detailed way in~\cite[Section 4]{BS22}. In order to keep this paper short, we focus on the statement and proof of the master relation itself. It is worth noticing that our proof follows from a direct application of the virtual localization formula for moduli spaces of (relative) stable maps to $\mathbb{P}^1$~\cite{GV}, while the previous proof in the Gorenstein quotient relied on a series of reduction steps to control the intersection of the master relation with generators of the tautological rings.

\subsection{Acknowledgments} X.~B. and S.~S. are supported by the Netherlands Organization for Scientific Research.

\section{The master relation} \label{sec:master-relation}

\subsection{Definition} We follow the exposition in~\cite{BLS}.

\subsubsection{Pre-stable star rooted trees}
We fix integers $m\geq 1$, $n\geq 1$, and $g\geq 0$, such that $2g-2+n+m>0$.  We denote by $\PSSRT_{g,n,m}$ the set of  {\em $n+m$-pre-stable star rooted trees of genus $g$}, that is the set of pre-stable graphs 
$$T=(V,H,\iota \colon H\to H,g\colon V\to \mathbb{Z}_{\geq 0}, H^{\iota}\simeq \{\sigma_1,\ldots, \sigma_{n+m}\} )$$
of genus $g$ with $n+m$ legs, satisfying the following constraints:
\begin{itemize}
	\item If $E(T)$ stands for the set of {\em edges} $H^\iota$, then the graph $(V,E)$ is a rooted tree, and all edges are between the root and another vertex (hence the term ``star''). 
	\item The legs $\sigma_{n+1},\dots,\sigma_{n+m}$ are attached to the root vertex, while $\sigma_{1},\dots,\sigma_{n}$ are attached to non-root vertices. In particular, the root vertex $v_r$ is uniquely determined by the rest of the data. 
	\item For each $v\in V(T)$, let $H(v)$ be the set of half-edges incident to $v$ (including the legs). The pre-stability condition means that $2g(v)-2+|H(v)| \geq 0$ for every $v\in V(T)$. 
	\item There is at least one leg attached to each vertex. 
\end{itemize} 

We will consider the ring $Q=\mathbb{Q}(a_1,\ldots,a_n)$ (where the $a_i$'s are formal variable). If  $T$ is a graph in $\PSSRT_{g,n,m}$, then each edge $e$ connects the root vertex $v_r$ to a non-root vertex that we denote $v_e$. We set $a(e)$ to be the sum of $a_i$'s associated with the legs attached to $v_e$. If a half-edge $h$  is the leg $\sigma_i$, then $a(h)$ stands for $a_i$. 

\subsubsection{Classes assigned to a tree}  We want to assign to $T$ a class $\Xi(T)$ in $R^*(\oM_{g,n+m})\otimes_{\QQ} Q[u,u^{-1}]$ (where $u$ is an extra formal variable used to control the degree). In order to construct this class, we first assign a class to each vertex.

Let $v$ be a vertex of $T$. We will consider the moduli space of curves $\oM_{g(v),|H(v)|}$. If $v$ is the root vertex $v_r$ we have a natural isomorphism $H(v_r) \cong E(T) \sqcup \{\sigma_{n+1},\dots,\sigma_{n+m}\}$ and we associate the first $|E(T)|$ marked points in the corresponding space $\oM_{g(v_r),|H(v_r)|}$ with the edges, and the remaining $m$ points with the legs $\{\sigma_{n+1},\dots,\sigma_{n+m}\}$. We set
\begin{align}
	\Psi(v_r)& \coloneqq \prod_{e\in E(T)} \frac{1}{1-a({e}) \psi_e} \in R^*(\oM_{g(v_r),|E(T)|+m})\otimes_{\QQ}Q
\end{align}
where $\psi_e$ stands for the $\psi$-class attach the half-edge part of the edge $e$. In the exceptional unstable case $g(v_r)=0, m=1, |E(T)|=1$, we formally assign to the root vertex the following class:
\begin{align}
	\Psi(v_r)\coloneqq a(e)^{-1}\in R^{-1}(\oM_{0,2})\otimes_{\QQ}Q,
\end{align}
where $R^{-1}(\oM_{0,2})$ is identified with $\QQ$ and the negative cohomological degree is just formally assigned to allow to treat this case non-exceptionally in what follows.

If $v$ is a non-root vertex $v_e\not=v_r$ for some edge $e$, the we associate the first $|H(v_e)|-1$ marked points in the space $\oM_{g(v_e),|H(v_e)|}$ with the legs attached to $v_e$, and the last marked point with the edge $e$ that connects $v_e$ to $v_r$. We set
\begin{align}
	\bD(v_{e})& \coloneqq \frac{\lambda_{g(v_{e})}\DR_{g(v_{e})}\big(a(h_1),\dots,a(h_{|H(v_{e})|-1}),-a(e)\big)}{1-a(e)\psi_{|H(v_e)|}}   \in R^*(\oM_{g(v_e),|H(v_e)|})\otimes_{\QQ}Q.
\end{align}
As $a(e) =a(h_1)+\cdots+a(h_{|H(v_e)|-1})$,  the class $\DR_{g(v_{e})}\big(a(h_1),\dots,a(h_{|H(v_{e})|-1}),-a(e)\big)$ is the corresponding double ramification cycle.  In the exceptional unstable case $g(v_e)=0, H(v_e)=2$ we formally assign to this vertex the following class:
\begin{align}
	\bD(v_{e})\coloneqq a(e)^{-1}\in R^{-1}(\oM_{0,2})\otimes_{\QQ}Q,
\end{align}
where, as it was for the root vertex,  the negative cohomological degree is just formally assigned to allow the treatment of this case non-exceptionally. With this system of notation, we set
\begin{align}
	\Xi(T)\coloneqq u^{2g-2+m} \left(\prod_{e\in E(T)} \frac{a(e)}{u}\right) (b_T)_* \left( \left(\sum_{d=-1}^\infty \frac{\Psi(v_r)_d}{(-u)^d}\right) \otimes\bigotimes_{e\in E(T)} \left( \sum_{d=-1}^\infty \frac{\bD(v_e)_d}{u^d} \right) \right),
\end{align} 
where $(b_T)_*\colon \bigotimes_{v \in V(T)} R^{*}(\oM_{g(v),|H(v)|})\to R^{*}(\oM_{g,n+m})$ is the boundary push-forward (composed with the contraction of unstable components) map extended by linearity with the coefficients in $Q[u,u^{-1}]$, and $\alpha_d$ stands for the projection to $R^{*}(\oM_{g(v),|H(v)|})$ of a tautological class.
Note that the resulting formula is a polynomial in $a_i$'s of degree bounded by $3g-3+n+m$.

\subsection{Statement and proof} For any $g\geq 0$, $m,n\geq 1$, $2g-2+n+m>0$,  we set 
\begin{align}
	\Xi^m_{g,n} \coloneqq \sum_{T\in \PSSRT_{g,n,m}} \Xi(T).
\end{align}
\begin{theorem} \label{conj:master-relation} We have
\begin{align}
	\Xi^m_{g,n} \in R^{*}(\oM_{g,n+m}) \otimes_{\QQ} \QQ[a_1,\dots,a_n,u].
\end{align}
(in other words, the coefficients of all negative degrees of $u$ vanish). 
\end{theorem}

\begin{proof}
Since $\Xi^m_{g,n}$ is a polynomial in $a_1,\dots,a_n$ of degree at most $3g-3+m+n$, it is sufficient to prove that $\Xi^m_{g,n}(a_1,\dots,a_n)$ is a polynomial in $u$ for each $a_1,\dots,a_n\in \ZZ_{>0}$. Note also that it is sufficient to consider the cohomology with the complex coefficients. Then the proof directly generalizes the argument in~\cite[Proof of Theorem 5.3]{BS22} along the lines of computation done~\cite[Section 3.2]{GKLS}.

We consider the moduli space $\oM_{g,n+m}(\mathbb{P}^1,a_1,\dots,a_n)$ of relative stable  maps to $(\mathbb{P}^1,\infty)$: the first $n$ marked points are mapped to $\infty$ with contact orders prescribed by the $a_i$'s. We consider the standard $\mathbb{C}^*$-action on $\mathbb{P}^1$ and induced action on $\oM_{g,m}(\mathbb{P}^1,a_1,\dots,a_n)$. Let $\pi\colon U\to \oM_{g,m}(\mathbb{P}^1,a_1,\dots,a_n)$ be the universal curve and $f\colon U\to \mathbb{P}^1$ the universal map. The curve $U$ is endowed with the $\mathbb{C}^*$-action making $f$ equivariant. For $i\in \{n+1,\ldots,n+m\}$ we denote  ${\rm ev}_i\colon \oM_{g,n+m}(\mathbb{P}^1,a_1,\dots,a_n)\to \mathbb{P}^1$ the evaluation morphism at the $i$-th marking.

We will consider the line bundle $\mathcal{O}(-1)\to \mathbb{P}^1$. We lift the $\mathbb{C}^*$-action to $\mathcal{O}(-1)\to \mathbb{P}^1$ with the fiber weights at $0$ and $\infty$ equal to $-1$ and $0$ respectively.  We set  
\begin{align}
	I^m_{g,n}\coloneqq  (-1)^{g+m}\ev^*_{n+1}([0])\cup \cdots \cup \ev^*_{n+m}([0])\cup e_{\mathbb{C}^*} (R^1\pi_*f^*\mathcal{O}(-1))
\end{align}
in  $H_*^{\mathbb{C}^*}(\oM_{g,n+m})\cong H_*(\oM_{g,n+m})\otimes_{\mathbb{C}} \mathbb{C}[u]$. In this expression, the class $[0]\in H^2_{\mathbb{C}^*}(\mathbb{P}^1)$ stands for the $\mathbb{C}^*$-equivariant class dual to the point $0$, while $e_{\mathbb{C}^*}$ is the $\mathbb{C}^*$-equivariant Euler class (here, the derived pushforward $R^1\pi_*f^*\mathcal{O}(-1)$ is a vector bundle of rank $g+a_1+\cdots+a_n-1$). 

Let $\epsilon \colon \oM_{g,m}(\mathbb{P}^1,a_1,\dots,a_n)\to \oM_{g,n+m}$ and $[\oM_{g,m}(\mathbb{P}^1,a_1,\dots,a_n)]^{\mathrm{vir}}$ denote the $\mathbb{C}^*$-equivariant virtual fundamental class of $\oM_{g,m}(\mathbb{P}^1,a_1,\dots,a_n)$. We will prove that 
\begin{align} \label{eq:epsilon-star}
	\Xi^{m}_{g,n}(a_1,\dots,a_n,-u)=\epsilon_*\left(I^m_{g,n} \cap [\oM_{g,m}(\mathbb{P}^1,a_1,\dots,a_n)]^{\mathrm{vir}} \right).
\end{align}
To this end, we need to apply the localization formula to~\eqref{eq:epsilon-star} along the lines of~\cite{GV} (see also a survey in~\cite[Appendix]{BS22}), and the computation of $e_{\mathbb{C}^*} (R^1\pi_*f^*\mathcal{O}(-1))$ performed in~\cite[Equation (3.13)]{GKLS}.  As the right-hand side of~\eqref{eq:epsilon-star} belongs to $H_*^{\mathbb{C}^*}(\oM_{g,n+m})\cong H_*(\oM_{g,n+m})\otimes_{\mathbb{C}} \mathbb{C}[u]$, the theorem follows.

The localization formula expresses $[\oM_{g,m}(\mathbb{P}^1,a_1,\dots,a_n)]^{\mathrm{vir}}$ as a sum over bipartite graphs with some vertices corresponding to the components mapped to $0$ and $\infty$. We allow these vertices to correspond to the unstable moduli spaces as well --- formally speaking, in such cases, a separate computation is required, but the final result can be written uniformly with the contributions of the unstable vertices introduced formally as we did above in the definition of $\Xi^m_{g,n}$.  Only part of these bipartite graphs contributes when we intersect the virtual fundamental class with $I^m_{g,n}$:
\begin{itemize}
\item The first $n$ legs are incident to vertices over $\infty$ because of the relative conditions, while the last $m$ legs are over $0$ because of the factors ${\rm ev}_i^*([0])$
\item Our specific choice of lift of the $\mathbb{C}^*$-action to $\mathcal{O}(-1)$ reduces the computation to an expression on star trees. Indeed, the expression of $e_{\mathbb{C}^*} (R^1\pi_*f^*\mathcal{O}(-1))$ has a factor $0^{|E(v)|-1}$ for each vertex $v$ over infinity, where $E(v)$ is the set of edges attached to $v$ (see \cite[Equation (3.13)]{GKLS})\footnote{This observation was used in \cite{Manin} and \cite{FaberPandharipande} to obtain the multiple cover formula for rigid rational curves in CY 3-folds.}.  Therefore, each vertex over $\infty$ should have exactly one edge (otherwise, the class vanishes). 
\end{itemize}
Hence, the bipartite graphs that we consider are, in fact, the graphs in $\PSSRT_{g,n,m}$: the root is the unique vertex over $0$ while the other vertices are over $\infty$. Note that the degree assignment at edges, which is part of the bi-partite graph data,  is uniquely determined by the relative conditions at infinity and the fact that the graphs are of compact type. Altogether, the virtual localization formula provides an expression of the form
\begin{align} \label{eq:epsilon-star1}
	\epsilon_*\left(I^m_{g,n} \cap [\oM_{g,m}(\mathbb{P}^1,a_1,\dots,a_n)]^{\mathrm{vir}} \right)= \sum_{T\in \PSSRT_{g,n,m}} (b_T)_*(\alpha_1(T)\alpha_2(T))
\end{align}
where $\alpha_{1}(T)$ corresponds to the contribution of $[\oM_{g,m}(\mathbb{P}^1,a_1,\dots,a_n)]^{\mathrm{vir}}$ from~\cite{GV} (see also a convenient reminder in~~\cite[Appendix]{BS22}), while $\alpha_{2}(T)$ corresponds to the contribution of $I_{m,n}$ deduced from~\cite[Equation (3.13)]{GKLS}.  The explicit expression of these classes is 
\begin{align}
 \alpha_1(T)=&\,-u^{-|E(T)|-1}\left(\sum_{i=0}^{g(v_r)} (-1)^i \lambda_i u^{g(v_r)-i} \prod_{e\in E(T)} \frac{a(e)^{a(e)+1}}{a(e)!} \frac{u^{-a(e)}}{ 1 - u^{-1} a(e) \psi_e}\right)\\
& \nonumber \,  \bigotimes_{e \in E(T)} \left(\frac{\DR_{g(v_{e})}\big(a(h_1),\dots,a(h_{|H(v_{e})|-1}),-a(e)\big)}{1+u^{-1}a(e)\psi_{|H(v_e)|}}\right),
\end{align}
and
\begin{align}
 \alpha_2(T)=&\,(-u)^{|E(T)|-1} (-1)^{g+m}u^m \left((-1)^{g(v_r)} \sum_{i=0}^{g(v_r)} \lambda_i u^{g(v_r)-i} \prod_{e \in E(T)}  \frac{a(e)!}{a(e)^{a(e)}} u^{a(e)-1} 
\right)\\
& \nonumber \,  \bigotimes_{e \in E(T)} (-1)^{g(v_e)}\lambda_{g(v_e)}.
\end{align}
Now we use Mumford's relation~\cite{Mumford}
\begin{align}
 \left(\sum_{i=0}^{g(v_r)} (-1)^i \lambda_i u^{g(v_r)-i}\right)\left( \sum_{i=0}^{g(v_r)} \lambda_i u^{g(v_r)-i}  \right) = u^{2g(v_r)}
\end{align} 
to deduce 
\begin{align}
 \alpha_1(T) \alpha_2(T)=&\, (-u)^{2g-2+m} \left(\prod_{e\in E(T)} \frac{a(e)}{-u} \frac{1}{1 - u^{-1} a(e) \psi_e}\right)\\
& \nonumber  \bigotimes_{e \in E(T)} \left(\frac{(-u)^{2g(v_r)}\lambda_{g(v_r)}\DR_{g(v_{e})}\big(a(h_1),\dots,a(h_{|H(v_{e})|-1}),-a(e)\big)}{1+u^{-1}a(e)\psi_{|H(v_e)|}}\right). 
\end{align}
Together with~\eqref{eq:epsilon-star1}, we indeed obtain the identity~\eqref{eq:epsilon-star}.
\end{proof}

\end{document}